\documentclass[12pt, reqno]{amsart}
\usepackage{amsmath, amsthm, amscd, amsfonts, amssymb, graphicx, color}
\usepackage[bookmarksnumbered, colorlinks, plainpages]{hyperref}
\hypersetup{colorlinks=true,linkcolor=red, anchorcolor=green, citecolor=cyan, urlcolor=red, filecolor=magenta, pdftoolbar=true}

\newtheorem{theorem}{Theorem}[section]
\newtheorem{lemma}[theorem]{Lemma}
\newtheorem{proposition}[theorem]{Proposition}
\newtheorem{corollary}[theorem]{Corollary}
\theoremstyle{definition}
\newtheorem{definition}[theorem]{Definition}

\theoremstyle{remark}
\newtheorem{remark}[theorem]{Remark}
\numberwithin{equation}{section}

\begin{document}

\setcounter{page}{1}

\title[Composition operators and convexity of their Berezin range]{Composition operators, convexity of their Berezin range and related questions}
\author[Athul Augustine, M. Garayev \MakeLowercase{and} P. Shankar]{Athul Augustine, M. Garayev \MakeLowercase{and} P. Shankar}

\address{Athul Augustine, Department of Mathematics, Cochin University of Science And Technology,  Ernakulam, Kerala - 682022, India. }
\email{\textcolor[rgb]{0.00,0.00,0.84}{athulaugus@gmail.com, athulaugus@cusat.ac.in}}

\address{M. Garayev, Department of Mathematics, College of Science , King Saud University, P.OBox 2455Riyadh 11451, Saudi Arabia }
\email{\textcolor[rgb]{0.00,0.00,0.84}{mgarayev@ksu.edu.sa}}

\address{P. Shankar, Department of Mathematics, Cochin University of Science And Technology,  Ernakulam, Kerala - 682022, India.}
\email{\textcolor[rgb]{0.00,0.00,0.84}{shankarsupy@gmail.com, shankarsupy@cusat.ac.in}}




\begin{abstract}
The Berezin range of a bounded operator $T$ acting on a reproducing kernel Hilbert space $\mathcal{H}$ is the set $\text{Ber}(T)$ := $\{\langle T\hat{k}_{x},\hat{k}_{x} \rangle_{\mathcal{H}} : x \in X\}$, where $\hat{k}_{x}$ is the normalized reproducing kernel for $\mathcal{H}$ at $x \in X$. In general, the Berezin range of an operator is not convex. In this paper, we discuss the convexity of range of the Berezin transforms. We characterize the convexity of the Berezin range for a class of composition operators acting on the Hardy space and the Bergman space of the unit disk. Also for so-called superquadratic functions, we prove the Berezin set mapping theorem for positive self-adjoint operators  $A$ on the reproducing kernel Hilbert space $\mathcal{H}(\Omega)$, namely we prove that $f(\mathrm{Ber}(\Phi(A)))=\mathrm{Ber}(\Phi(f(A)))$, where $\Phi:\mathcal{B}%
\left(  \mathcal{H}\left(  \Omega\right)  \right)  \mathcal{\rightarrow
}\mathcal{B}\left(  \mathcal{K(}Q\mathcal{)}\right)  $ is a normalized
positive linear map.

\end{abstract}
\maketitle

\section{Introduction}
The \textit{numerical range} of a bounded linear operator $A$ on a Hilbert space $\mathcal{H}$ is defined as
	$$W(A):= \{\langle Au,u\rangle : \|u\| = 1 \}.$$
By the Toeplitz-Hausdorff theorem \cite{toeplitz}
the numerical range of a linear operator on a Hilbert space is always convex.
	Given a set $X$, we say that $\mathcal{H}$ is a \textit{reproducing kernel Hilbert space} (RKHS) on $X$ over $\mathbb{C}$, if $\mathcal{H}$ is a vector subspace of the set of all functions from $X$ to $\mathbb{C}$, $\mathcal{H}$ endowed with an inner product, making it into a Hilbert space and for every $y\in X$, the linear evaluation functional, $E_{y}: \mathcal{H}\rightarrow \mathbb{C}$, defined by $E_{y} = f(y)$ is bounded. Throughout this paper, we work in the setting of reproducing kernel Hilbert spaces.

Let $\mathcal{H}$ be an RKHS, then by the Riesz representation theorem, there is a unique element $k_x \in \mathcal{H}$ such that $\langle f,k_x\rangle_{\mathcal{H}} = E_{x}(f) = f(x)$ for each $x \in X$ and all $f\in \mathcal{H}$. The element $k_x$ is called the  \textit{reproducing kernel} at $x$. We denote the \textit{normalized reproducing kernel} at $x$ as $\hat{k}_x = k_x/{\| k_x \|}_{\mathcal{H}}.$ From now onwards, $\mathcal{H}$ denotes a reproducing kernel Hilbert space on some set $X$. For more background about RKHS, we refer the reader to look at \cite{paulsen2016introduction}.

For a bounded linear operator $T$ acting on $\mathcal{H}$, the \textit{Berezin range} of $T$ is defined as 
$$\text{Ber}(T) := \{\langle T\hat{k}_{x},\hat{k}_{x} \rangle_{\mathcal{H}} : x \in X\},$$ 
where $\hat{k}_{x}$ is the normalized reproducing kernel for $\mathcal{H}$ at $x \in X.$ 

The Berezin transform was first introduced by Berezin in \cite{berezin1972covariant}. The Berezin transform plays a crucial role in operator theory. Many fundamental properties of basic operators are encrypted in the Berezin transforms. Karaev in \cite{karaev2006} formally introduced the Berezin set and Berezin number. 

The convexity of the Berezin range is the main focus of this paper. Given a bounded operator $T$ acting on an RKHS $\mathcal{H}$, is $\text{Ber}(T)$ convex ? By the Toeplitz-Hausdorff theorem, the numerical range of an operator is always convex \cite{toeplitz}. It is easy to observe that the Berezin range of an operator $T$ is always a subset of the numerical range of $T$. In general, the Berezin range of an operator need not be convex. Karaev  \cite{karaev2013reproducing} initiated the study of the geometry of the Berezin range. He \cite[Section 2.1]{karaev2013reproducing} showed that the Berezin range of the Model operator $M_{z^n}$ on the Model space is convex. This is the first result in the geometric or set-theoretic viewpoint of the Berezin range. This motivated Cowen and Felder to explore more about the convexity of the Berezin range.
Cowen and Felder \cite{cowen22} characterized the convexity of the Berezin range for matrices, multiplication operators on RKHS, and a class of composition operators acting on the Hardy space of the unit disc. Cowen and Felder raised several open questions that followed naturally from \cite{cowen22}. Here we address some of these questions.

This paper mainly focuses on the question \cite[Question 5.5]{cowen22}. Given a class of concrete operators acting on the Bergman space, what can be said about the convexity of the Berezin range of these operators? Karaev \cite{15} explored relation between the Berezin set $\mathrm{Ber}(T)$  of an operator $T$ and its spectrum $\sigma(T)$ for some concrete operators. The spectral mapping theorem of the form $\varphi
(\sigma(T))=\sigma(\varphi(T))$ is an important tool in studying many problems in operator theory. Another important question of our interest is as follows. Can we have analogous results of the spectral mapping theorem for Berezin sets?

The paper is organized as follows. In Section 2, we discuss basic notions and provide the necessary results. Section 3 proves that the Berezin range of an infinite dimensional matrix $A$ is convex if and only if $A$ has a constant diagonal. In Section 4, we characterize the convexity of the Berezin range of composition operators on the Hardy space for the general case of the elliptic symbol and a particular case of the automorphic symbol. In Section 5, we characterize the convexity of the Berezin range of composition operators on the Bergman space for the general case of elliptic symbol, a particular case of automorphic symbol, and the Blaschke factor. In Section 6, we prove the Berezin set mapping theorem for some self-adjoint operators on the reproducing kernel Hilbert space using superquadratic functions.

\section{Preliminaries}
A well-known example of an RKHS is the classical \textit{Hilbert Hardy space} \cite{paulsen2016introduction} on the unit disc $\mathbb{D}$,
$$H^{2}(\mathbb{D}) = \left\lbrace f(z) = \displaystyle\sum_{n\geq 0}a_{n}z^{n} \in \text{Hol}(\mathbb{D}) :\displaystyle\sum_{n\geq 0}|a_{n}|^{2} < \infty \right\rbrace$$
where Hol$(\mathbb{D})$ denotes the collection of holomorphic functions on $\mathbb{D}$. For $f(z)=\sum_{n\geq 0}a_{n}z^{n}$ and $g(z)=\sum_{n\geq 0}b_{n}z^{n}$ be the elements in $H^2$, then their inner product is defined as $\langle f,g \rangle = \sum_{n\geq 0}a_{n}\overline{b_n}$. $H^{2}(\mathbb{D})$ is a RKHS and the reproducing kernel for $H^{2}(\mathbb{D})$ is given by
 $$k_{w}(z) = \frac{1}{1 - \bar{w}z},\quad z,w \in \mathbb{D}.$$
 
 Another RKHS, we study here is the \textit{Bergman space} \cite{paulsen2016introduction} on the unit disc $\mathbb{D}$
 $$A^2(\mathbb{D}) = \left\lbrace f\in \text{Hol}(\mathbb{D}): \int_{\mathbb{D}} |f(z)|^2dV(z) < \infty\right\rbrace. $$
 where $dV$ is the normalized area measure on $\mathbb{D}$. The reproducing kernel for $A^2(\mathbb{D})$ is 
 $$ k_w(z) = \frac{1}{(1-\bar{w}z)^2},\quad z,w \in \mathbb{D}.$$

\begin{definition}
	Let $\mathcal{H}$ be an RKHS on a set $X$ and let $T$ be a bounded linear operator on $\mathcal{H}$.
	For $x \in X$,
	\begin{enumerate}
		\item The \textit{Berezin transform} of $T$ at x is
		\begin{center}
			$\widetilde{T}(x) := \langle T\hat{k}_{x},\hat{k}_{x} \rangle_{\mathcal{H}}$.
		\end{center}
		\item The \textit{Berezin range} of $T$ is
		\begin{center}
			$ \text{Ber}(T)$ := $\{\langle T\hat{k}_{x},\hat{k}_{x} \rangle_{\mathcal{H}} : x \in X\}$.
		\end{center}
		\item The \textit{Berezin radius} of $T$ (or \textit{Berezin number} of $T$) is 
		\begin{center}
			$\text{ber}(T) := \displaystyle\sup_{x \in X} |\widetilde{T}(x)|$.
		\end{center}
	\end{enumerate}
\end{definition}
Cowen and Felder \cite[Section 2]{cowen22} documented many interesting results regarding the Berezin transform. The Berezin transform is used to study the invertibility and compactness of many well-known operators on various RKHS. Before the Cowen and Felder paper \cite{cowen22}, there was no study about the Berezin range in a set-theoretic or geometric viewpoint other than the examples due to Karaev \cite[Section 2.1]{karaev2013reproducing}. This motivated Cowen and Felder to prove the following results on the convexity of the Berezin range. 

Let $\mathcal{H}$ be a Hilbert space of functions, then 
$$\text{Mult}(\mathcal{H}):= \{g\in \mathcal{H} : gf\in \mathcal{H}~~ \text{for all}~~ f \in \mathcal{H}\}.$$ 
For $g\in \text{Mult}(\mathcal{H}),$ we define the  multiplication operator $M_g$ on $\mathcal{H}$ by $M_gf=gf$.
\begin{theorem}\cite[Proposition 3.2]{cowen22} Let $\mathcal{H}$ be an RKHS on a set $X$ and $g\in \text{Mult}(\mathcal{H})$. Then the Berezin range of $M_g$ is convex if and only if $g(X)$ is convex.
\end{theorem}

Consider a complex-valued function $\phi : X \rightarrow X$ and a composition operator $C_\phi$ acting on a space of functions defined on $X$ by
$$C_\phi f := f \circ \phi.$$

\begin{theorem}\label{elli}\cite[Theorem 4.1]{cowen22} Let $\zeta \in \mathbb{T}$ and $\phi(z) = \zeta z$. Then the Berezin range of $C_\phi$ acting on $H^2$ is convex if and only if $\zeta=1$ or $\zeta = -1$.
\end{theorem}

\begin{theorem}\cite[Theorem 4.5]{cowen22} Let $\alpha \in \mathbb{D}$ and $\phi_\alpha = \frac{z-\alpha}{1-\overline{\alpha}z}$. Then the Berezin range of $C_\phi$ acting on $H^2$ is convex if and only if $\alpha=0.$
\end{theorem}

For the sake of completeness, here we discuss the first result about the convexity of the Berezin range by Karaev in  \cite{karaev2013reproducing}.

  Let $H^\infty(\mathbb{D})$ denote the set of bounded analytic functions on $\mathbb{D}$. A function $u \in H^\infty(\mathbb{D})$ is said to be inner if $|u(z)|=1$ a.e on $\mathbb{T}$. Suppose $\theta$ is an inner function. We define the corresponding Model space by the formula
\begin{center}
	$K_\theta:=H^2 \Theta \theta H^2$.
\end{center}
The Model operator $M_\theta$ on the Model space $K_\theta$ is defined as
\begin{center}
	$M_\theta := P_\theta S| K_\theta$
\end{center}
  where $S$ is the one-sided shift operator on $H^2,$ i.e., $Sf=zf$, $f\in H^2$ and $P_\theta$ is an orthogonal projection of $H^2$ onto $K_\theta$.
  The normalized reproducing kernel of the model space $K_\theta$ is the function

	$$\hat{k}_{\theta,\lambda}(z) =\left(\frac{1-|\lambda|^2}{1-|\theta(\lambda)|^2}\right)^{\frac{1}{2}}\frac{1-\overline{\theta(\lambda)}\theta(z)}{1-\overline{\lambda}z}.$$
  
For $\theta=z^n$, the Model space $K_{z^n}$ is
 $$K_{z^n} = H^2\Theta z^n H^2$$
and the corresponding Model operator $M_{z^n}$ is
$$M_{z^n}= P_{z^n}S| K_{z^n}.$$

\begin{theorem}\cite[Section 2.1]{karaev2013reproducing}
	The Berezin range of the Model operator $M_{z^n}$ is $\mathbb{D}_{\frac{n-1}{n}}$, which is always convex and $\text{ber}(M_{z^n}) = \frac{n-1}{n}.$
\end{theorem}

\section{Infinite dimensional matrices}

  	The numerical range is invariant under unitary equivalence, but this is not the case for the Berezin range. Two matrices that are unitarily equivalent may not have the same Berezin range.

For
  		$A = \begin{bmatrix}
  		1 & 0\\
  		0 & 2
  	\end{bmatrix}$ and $C = \frac{1}{2} \begin{bmatrix}
  		3 & 1\\
  		1 & 3
  	\end{bmatrix}$ are unitarily equivalent matrices with respect to the unitary matrix $U = \frac{1}{\sqrt{2}} \begin{bmatrix}
  		1 & -1\\
  		1 & 1
  	\end{bmatrix}$. It is easy to observe that the Berezin range of A and C are not equal ($\text{Ber}(A) = \{1,2\}$ and $\text{Ber}(C) = \{\frac{3}{2}\}$). 
  	
Since the Berezin range is not invariant under unitary equivalence, many fundamental properties of the numerical range are not carrying forward to the Berezin range.

For a finite-dimensional $n \times n$ matrix $A$ with complex entries, under the standard inner product for $\mathbb{C}^n$, the Berezin range of $A$ is convex if and only if $A$ has constant diagonal \cite{cowen22}. Now we extend this result to infinite matrices.
  
We consider $\mathnormal{l}^2$ as the set of all functions mapping $X \rightarrow \mathbb{C}$ by $v(j) = v_{j}$ and $X = \{1,2,......\}$. Then $\mathnormal{l}^2$ is a RKHS with kernel $k_{j} = e_{j}$, the $j^{th}$ standard basis vector and $k_{j} = \hat{k}_{j}$ for each j =1,2....
  For any complex matrix $A=(a_{jk})_{j,k=1}^\infty$, we have
  \begin{center}
  	$B(A) = \{\langle Ae_j,e_j\rangle : j = 1,2,..\}$\\
  	$B(A)= \{a_{jj} : j=1,2..\}$\\
  	
  \end{center}
  So the Berezin range of the complex matrix $A$ is the collection of diagonal elements of $A$. It follows that $B(A)$ is convex if and only if the diagonal elements of matrix $A$ are all equal.
  \begin{proposition}
  	Let $A=(a_{jk})_{j,k=1}^\infty$ be an infinite matrix with complex entries. Under the standard inner product for $\mathnormal{l}^2$, the Berezin range of $A$ is convex if and only if $A$ has a constant diagonal.
  \end{proposition}

\section{ Compositon operator on Hardy space}
 
\subsection{Elliptic symbol}
Cowen and Felder \cite{cowen22} characterized the convexity of the Berezin range of the composition operator $C_\phi$ where $\phi(z) = \zeta z$, $\zeta\in \mathbb{T}$ and $z \in \mathbb{D}$ on the RKHS $H^2(\mathbb{D})$. Here we characterize a more general composition operator.

For $\alpha \in \mathbb{\overline{D}}$ and $z\in\mathbb{D}$, let $\phi(z) = \alpha z$. Acting on $H^2$, we have
  \begin{center}
  	\begin{equation*}
  		\begin{split}
  			\widetilde{C_{\phi}}(z) &=\langle C_{\phi}\hat{k}_{z},\hat{k}_{z}\rangle\\
  			&=(1-|z|^2)\langle C_{\phi}{k_{z}},{k_{z}}\rangle\\
  			&=(1-|z|^2)k_{z}(\phi(z))\\
  			&=\frac{1 - |z|^2}{1 - |z|^{2}\alpha}.
  		\end{split}
  	\end{equation*}
  \end{center}
  The Berezin range of these operators is not always convex, as we will see in the example below. Here we try to find the values of $\alpha$ for which $\text{Ber}(C_\phi)$ is convex.
  \begin{figure}[h]
  	\caption{$\text{Ber}(C_\phi)$ on $H^2$ for $\alpha=-0.5$ (left, apparently convex) and $\alpha=0.25+0.25i$(right, apparently not convex).}
  	\includegraphics[scale=2]{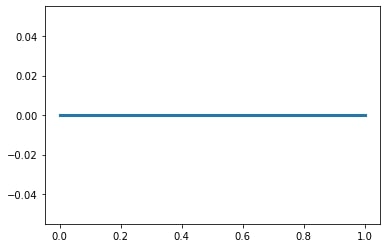}
  	\includegraphics[scale=2]{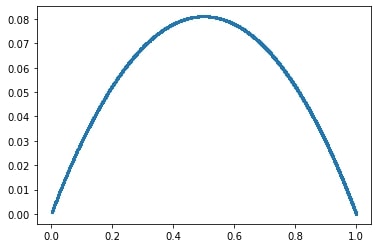}
  \end{figure}

\begin{theorem}\label{zeta}
Let $\alpha\in \overline{\mathbb{D}}$ and $\phi(z) = \alpha z$. Then the Berezin range of $C_\phi$ acting on $H^2$ is convex if and only if $-1\leq \alpha \leq 1.$
\end{theorem}

\begin{proof}
	Let $\alpha$ $\in$ $\overline{\mathbb{D}}$ and $z\in\mathbb{D}$. Put $z=re^{i\theta}$ with $0\leq r < 1$. Then,
	\begin{center}
		\begin{equation*}
			\begin{split}
				\widetilde{C_{\phi}}(z) &=\langle C_{\phi}\hat{k}_{z},\hat{k}_{z}\rangle\\
				&=\frac{1 - |z|^2}{1 - |z|^{2}\alpha}\\
				&=\frac{1 - r^2}{1 - r^{2}\alpha}
			\end{split}
		\end{equation*}

	\end{center}

If $\alpha = 1$ then 
\begin{center}
	$\widetilde{C_{\phi}}(re^{i\theta}) =\frac{1 - r^2}{1 - r^{2}\alpha} = 1$.
\end{center}
 So $\text{Ber}(C_\phi) = \{1\}$, which is convex. Similarly if $-1\leq \alpha < 1$, we have
 \begin{center}
 	$\widetilde{C_{\phi}}(re^{i\theta}) =\frac{1 - r^2}{1 - r^{2}\alpha}$.\\
 \end{center}
  So $\text{Ber}(C_\phi) = \{\frac{1 - r^2}{1 - r^{2}\alpha} : r\in[0,1)\} = (0,1]$ which is also convex.

	Conversely, suppose that $\text{Ber}(C_\phi)$ is convex. We have
	\begin{center}
		$\widetilde{C_{\phi}}(re^{i\theta}) =\frac{1 - r^2}{1 - r^{2}\alpha}.$
	\end{center} 
Here $\widetilde{C_{\phi}}$ is a function which is independent of $\theta$. Therefore $\text{Ber}(C_\phi)$ is a path in $\mathbb{C}$.  So if $\text{Ber}(C_\phi)$ is convex, it must be either a point or a line segment. It is easy to observe that $\text{Ber}(C_\phi)$ is a point if and only if $\alpha = 1$, so assume $\text{Ber}(C_\phi)$ is a line segment. Note that $\widetilde{C_{\phi}}(0) = 1$ and that $\lim_{r\rightarrow 1^-}\widetilde{C_{\phi}}(re^{i\theta}) = 0$. This tells us that $\text{Ber}(C_{\phi})$ must be a line segment passing through the point 1 and approaching the origin. Consequently, we must have the imaginary part of $\text{Ber}(C_{\phi})$ to be zero, which can happen if and only if the imaginary part of $\alpha$ is zero. Since $\alpha \in \overline{\mathbb{D}}$, we have $-1 \leq \alpha \leq 1$. 
 So if $\text{Ber}(C_\phi)$ is convex then $-1 \leq \alpha \leq 1$.
\end{proof}

It is easy to observe that when $\alpha \in \mathbb{T}$ and $\phi(z)=\alpha z$, the Berezin range of $C_\phi$ acting on $H^2$ is convex if and only if $ \alpha = 1$ or $-1$. So Theorem \ref{elli} can be considered a corollary to the above theorem.

Every automorphism $\phi$ of the unit disc $\mathbb{D}$ is	of the form $\phi(z) = \frac{az+b}{\overline{b}z+\bar{a}}$ where $a,b\in \mathbb{C}$ and $|a|^2 - |b|^2=1$.
For $a,b \in \mathbb{C}$ and $z \in \mathbb{D}$, consider the composition operator $C_{\phi}$ acting on $H^2$. We have
\begin{center}
	\begin{equation*}
		\begin{split}
			\widetilde{C_{\phi}}(z) &=\langle C_{\phi}\hat{k}_{z},\hat{k}_{z}\rangle\\
			&=(1-|z|^2)\langle C_{\phi}{k_{z}},{k_{z}}\rangle\\
			&=(1-|z|^2)k_{z}(\phi(z))\\
			&=\frac{1 - |z|^2}{1 - \bar{z}\phi(z)}
		\end{split}
	\end{equation*}
\end{center}
The Berezin range of these operators is not always convex, as we will see in the example below. Here we try to find the values for which $\text{Ber}(C_\phi)$ is convex. The following remark is a consequence of Theorem \ref{elli}.
\begin{figure}[h]
	\caption{$\text{Ber}(C_\phi)$ on $H^2$ for $\phi(z) = \frac{az+b}{\overline{b}z+\bar{a}}$ with $a=i$ and $b=0$ (left, apparently convex) and $\phi(z) = \frac{az+b}{\overline{b}z+\bar{a}}$ with $a=e^{\frac{i\pi}{12}}$ and $b=0$(right, apparently not convex).}
	\includegraphics[scale=2]{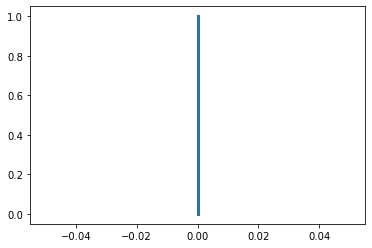}
	\includegraphics[scale=0.48]{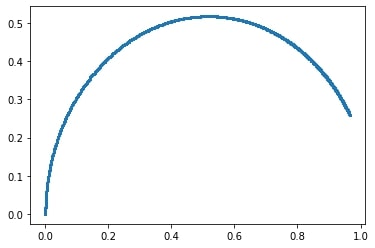}
\end{figure}

  \begin{remark}
  	For $b=0$, the Berezin range of the composition operator $C_\phi$, $\phi(z) = \frac{az+b}{\overline{b}z+\bar{a}}$ where $a,b\in \mathbb{C}$ and $|a|^2 - |b|^2=1$  is convex if and only if $a=1,-1,i,-i$.
  \end{remark}

\section{Composition operator on Bergman space}
Let $\phi$ be a complex-valued function $\phi : \mathbb{D}\longrightarrow \mathbb{D}$. The Berezin transform of a composition operator $C_\phi$ acting  on the Bergman space $A^2(\mathbb{D})$ is as follows:
\begin{center}
  	\begin{equation*}
  		\begin{split}
  			\widetilde{C_{\phi}}(z) &=\langle C_{\phi}\hat{k}_{z},\hat{k}_{z}\rangle\\
  			&=(1-|z|^2)^2\langle C_{\phi}{k_{z}},{k_{z}}\rangle\\
  			&=(1-|z|^2)^2k_{z}(\phi(z))\\
  			&=\frac{(1 - |z|^2)^2}{(1 - \overline{z}\phi(z))^2}.
  		\end{split}
  	\end{equation*}
  \end{center}
 \subsection{Elliptic Symbol}
For $\alpha \in \overline{\mathbb{D}}$ and $z\in\mathbb{D}$, let
 $\phi(z) = \alpha z$
  and consider the composition operator
$C_{\phi}f = f\circ \phi$
  acting on $A^2(\mathbb{D})$. We have
  \begin{center}
  	\begin{equation*}
  		\begin{split}
  			\widetilde{C_{\phi}}(z) &=\langle C_{\phi}\hat{k}_{z},\hat{k}_{z}\rangle\\
  			&=\frac{(1 - |z|^2)^2}{(1 - \overline{z}\phi(z))^2}\\
  			&=\frac{(1 - |z|^2)^2}{(1 - |z|^2\alpha)^2}.
  		\end{split}
  	\end{equation*}
  \end{center}
  The Berezin range of these operators is not always convex, as we will see in the example below. Here we try to find the values of $\alpha$ for which $\text{Ber}(C_\phi)$ is convex.

  \begin{figure}[h]
  	\caption{$\text{Ber}(C_\phi)$ on $H^2$ for $\alpha=-1$ (left, apparently convex) and $\alpha=i$(right, apparently not convex).}
  	\includegraphics[scale=2]{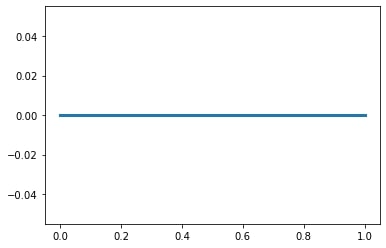}
  	\includegraphics[scale=2]{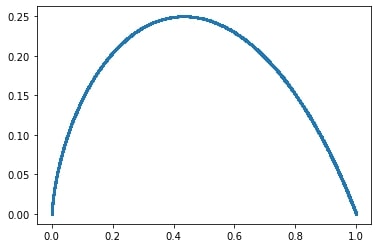}
  \end{figure}
\newpage
\begin{theorem}\label{elliptic}
Let $\alpha\in \overline{\mathbb{D}}$ and $\phi(z) = \alpha z$. Then the Berezin range of $C_\phi$ acting on $A^2(\mathbb{D})$ is convex if and only if $-1\leq\alpha\leq 1.$
\end{theorem}

\begin{proof}
	Let $\alpha$ $\in$ $\overline{\mathbb{D}}$ and $z\in\mathbb{D}$. Put $z=re^{i\theta}$ with $0\leq r < 1$. Then,
	\begin{center}
		\begin{equation*}
			\begin{split}
				\widetilde{C_{\phi}}(z) &=\langle C_{\phi}\hat{k}_{z},\hat{k}_{z}\rangle\\
				&=\frac{(1 - |z|^2)^2}{(1 - |z|^2\alpha)^2}\\
				&=\frac{(1 - r^2)^2}{(1 - r^2\alpha)^2}.
			\end{split}
		\end{equation*}

	\end{center}

If $\alpha = 1$ then 

	$$\widetilde{C_{\phi}}(re^{i\theta}) =\frac{(1 - r^2)^2}{(1 - r^2\alpha)^2} = 1.$$

 So $\text{Ber}(C_\phi) = \{1\}$, which is convex. Similarly if $\alpha \in [-1,1)$, we have
 
 	$$\widetilde{C_{\phi}}(re^{i\theta}) =\frac{(1 - r^2)^2}{(1 - r^2\alpha)^2}.$$
 
  So $\text{Ber}(C_\phi) = \{\frac{(1 - r^2)^2}{(1 - r^2\alpha)^2} : r\in[0,1)\} = (0,1]$ which is also convex.

	Conversely, suppose that $\text{Ber}(C_\phi)$ is convex. We have
	\begin{center}
		$\widetilde{C_{\phi}}(re^{i\theta}) =\frac{(1 - r^2)^2}{(1 - r^2\alpha)^2}$
	\end{center} 
which is a function independent of $\theta$. So if $\text{Ber}(C_\phi)$ is convex, it must be either a point or a line segment. It is immediate that $\text{Ber}(C_\phi)$ is a point if and only if $\alpha = 1$, so let us assume $\text{Ber}(C_\phi)$ is a line segment. Note that $\widetilde{C_{\phi}}(0) = 1$ and that $\lim_{r\rightarrow 1^-}\widetilde{C_{\phi}}(re^{i\theta}) = 0$. This tells us that $\text{Ber}(C_{\phi})$ must be a line segment passing through the point 1 and approaching the origin. Consequently, we must have the imaginary part of $\text{Ber}(C_{\phi})$ to be zero, which can happen if and only if the imaginary part of $\alpha$ is zero. Since $\alpha \in \overline{\mathbb{D}}$, we have $-1\leq \alpha \leq 1$.\\
\end{proof}

Now for $\alpha \in \mathbb{T}$ and $z\in\mathbb{D}$, consider the elliptic symbol 
  	$\phi(z) = \alpha z$ and the composition operator $C_{\phi}f = f\circ \phi.$
  Acting on $A^2(\mathbb{D})$, we have
  \begin{center}
  	\begin{equation*}
  		\begin{split}
  			\widetilde{C_{\phi}}(z) &=\langle C_{\phi}\hat{k}_{z},\hat{k}_{z}\rangle\\
  			&=\frac{(1 - |z|^2)^2}{(1 - |z|^2\alpha)^2}.
  		\end{split}
  	\end{equation*}
  \end{center}
  
\begin{corollary}
	The Berezin range of $C_\phi$ acting on $A^2(\mathbb{D})$ is convex if and only if $\alpha =1$ or $-1.$
\end{corollary}

Consider the automorphism of the unit disc $\phi(z) = \frac{az+b}{\overline{b}z+\bar{a}}$ where $a,b\in \mathbb{C}$ and $|a|^2 - |b|^2=1$. For $a,b \in \mathbb{C}$ and $z \in \mathbb{D}$, consider the composition operator $C_{\phi}$ acting on $A^2(\mathbb{D})$. We have
$$\widetilde{C_{\phi}}(z)=\frac{(1 - |z|^2)^2}{(1 - \overline{z}\phi(z))^2}.$$

The Berezin range of these operators is not always convex, as we will see in the example below. Here we try to find the values for which $\text{Ber}(C_\phi)$ is convex. 

\begin{figure}[h]
	\caption{$\text{Ber}(C_\phi)$ on $H^2$ for $\phi(z) = \frac{az+b}{\overline{b}z+\bar{a}}$ with $a=i$ and $b=0$ (left, apparently convex) and $\phi(z) = \frac{az+b}{\overline{b}z+\bar{a}}$ with $a=e^{\frac{i\pi}{12}}$ and $b=0$(right, apparently not convex).}
	\includegraphics[scale=2]{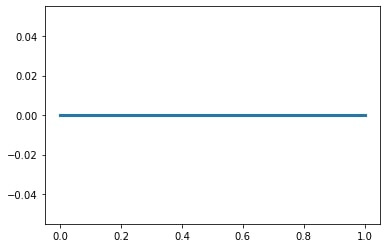}
	\includegraphics[scale=2]{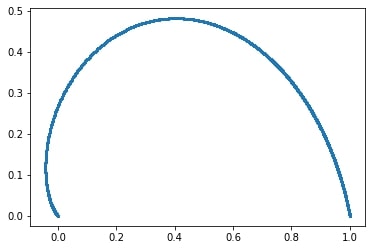}
\end{figure}
 The following remark is a consequence of Theorem \ref{elliptic}.
  \begin{remark}
  	For $b=0$, $\text{Ber}(C_\phi)$ is convex if and only if $a=1,-1,i,-i$.
  \end{remark}

\subsection{Blaschke Factor}
Consider the automorphism of the unit disc known as the Blaschke factor
$$\phi_\alpha(z) = \frac{z-\alpha}{1-\overline{\alpha}z}
$$
where $\alpha\in \mathbb{D}$	  and the composition operator $C_{\phi_\alpha}$ acting on $A^2(\mathbb{D})$. We have
  \begin{center}
  	\begin{equation*}
  		\begin{split}
  			\widetilde{C}_{\phi_\alpha}(z) &=\langle C_{\phi}\hat{k}_{z},\hat{k}_{z}\rangle\\
  			&=(1-|z|^2)^2\langle C_{\phi_\alpha}{k_{z}},{k_{z}}\rangle\\
  			&=(1-|z|^2)^2k_{z}(\phi_\alpha(z))\\
  			&=\frac{(1 - |z|^2)^2}{(1 - \overline{z}\phi_\alpha(z))^2}.
  	\end{split}
  	\end{equation*}
  \end{center}
  The Berezin range of these operators is not always convex, as we will see in the example below. Here we try to find the values of $\alpha$ for which $\text{Ber}(C_\phi)$ is convex.
  \begin{figure}[h]
  	\caption{$\text{Ber}(C_{\phi_\alpha})$ on $L_a^2$ for $\alpha=0.5$.}
  	\includegraphics[scale=3]{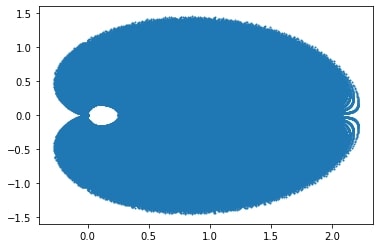}
  \end{figure}

\begin{lemma}\label{real}
	On $A^2(\mathbb{D})$, the real and imaginary parts of $\widetilde{C}_{\phi_\alpha}$ are given by 
	\begin{equation*}
	\begin{aligned}
	Re\{\widetilde{C}_{\phi_\alpha}(z)\} &= k_{\alpha,z}^2\left[(1-|z|^2)(1-Re\{\bar{\alpha}z\})+2(Im\{\bar{\alpha}z\})^2\right]^2\\ 
	&\qquad -k_{\alpha,z}^2 \left[Im\{\bar{\alpha}z\}(1+|z|^2 - 2Re\{\bar{\alpha}z\})\right]^2
	\end{aligned}
	\end{equation*}
	
and

\begin{equation*}
	\begin{aligned}
	Im\{\widetilde{C}_{\phi_\alpha}(z)\} &= 2k_{\alpha,z}^2 \left[(1-|z|^2)(1-Re\{\bar{\alpha}z\})+2(Im\{\bar{\alpha}z\})^2\right]\\ 
	& \qquad\times\left[Im\{\bar{\alpha}z\}(1+|z|^2 - 2Re\{\bar{\alpha}z\})\right]
	\end{aligned}
	\end{equation*}

where
$$k_{\alpha,z} = \frac{(1-|z|^2)}{(1-|z|^2 + 2iIm\{\alpha\bar{z}\})}.$$
\end{lemma}

\begin{proof}
We will compute $\widetilde{C}_{\phi_\alpha}$:
\begin{center}
\begin{equation*}
  		\begin{split}
  			\widetilde{C}_{\phi_\alpha}(z) &=\left[\frac{1-|z|^2}{1-\overline{z}\phi_\alpha(z)}\right]^2\\
  			&=\left[\frac{(1-|z|^2)(1-\overline{\alpha}z)}{1-\overline{\alpha}z-\overline{z}(z-\alpha)}\right]^2\\
  			&=\left[\frac{(1-|z|^2)(1-\overline{\alpha}z)}{1-|z|^2+2iIm\{\alpha\overline{z}\}}\right]^2.
    	\end{split}
\end{equation*}
\end{center}
 Multiplying the complex conjugate in the denominator, we get
 \begin{center}
\begin{equation*}
  		\begin{split}
  			\widetilde{C}_{\phi_\alpha}(z) &=\left[k_{\alpha,z}(1-\overline{\alpha}z)(1-|z|^2-2iIm\{\alpha\overline{z}\})\right]^2\\
  			&=\left[k_{\alpha,z}(1-|z|^2+2iIm\{\overline{\alpha}z\}-\overline{\alpha}z(1-|z|^2)+2iIm\{\alpha\overline{z}\}\overline{\alpha}z)\right]^2\\
  			&=k_{\alpha,z}^2[1-|z|^2+2iIm\{\overline{\alpha}z\}-(Re\{\overline{\alpha}z\}+iIm\{\overline{\alpha}z\})(1-|z|^2)\\
  			&\qquad -2iIm\{\overline{\alpha}z\}(Re\{\overline{\alpha}z\}+iIm\{\overline{\alpha}z\})]^2\\
  		   &=k_{\alpha,z}^2[1-|z|^2+2iIm\{\overline{\alpha}z\}- (1-|z|^2)Re\{\overline{\alpha}z\} - i(1-|z|^2)Im\{\overline{\alpha}z\} \\
  			& \qquad -2iIm\{\overline{\alpha}z\}Re\{\overline{\alpha}z\} +2(Im\{\overline{\alpha}z\})^2]^2\\
  			&=k_{\alpha,z}^2[1-|z|^2 - (1-|z|^2)Re\{\overline{\alpha}z\} +2(Im\{\overline{\alpha}z\})^2)\\
  			&\qquad +i(2Im\{\overline{\alpha}z\} - (1-|z|^2)Im\{\overline{\alpha}z\} - 2Im\{\overline{\alpha}z\}Re\{\overline{\alpha}z\}]^2\\
  			&=k_{\alpha,z}^2[(1-|z|^2)(1-Re\{\overline{\alpha}z\}) + 2(Im\{\overline{\alpha}z\})^2\\
  			&\qquad + iIm\{\overline{\alpha}z\}(1+|z|^2-2Re\{\overline{\alpha}z\})]^2
    	\end{split}
\end{equation*}
\end{center}
Squaring and combining the real and imaginary parts, we get
\begin{center}
  	\begin{equation*}
  		\begin{split}
  			\widetilde{C}_{\phi_\alpha}(z) &=k_{\alpha,z}^2\left[(1-|z|^2)(1-Re\{\bar{\alpha}z\})+2(Im\{\bar{\alpha}z\})^2\right]^2 \\
  			&\qquad -k_{\alpha,z}^2 \left[Im\{\bar{\alpha}z\}(1+|z|^2 - 2Re\{\bar{\alpha}z\})\right]^2 \\
  			&\quad+ik_{\alpha,z}^2 \left[(1-|z|^2)(1-Re\{\bar{\alpha}z\})+2(Im\{\bar{\alpha}z\})^2\right]\\
  			&\qquad\times\left[Im\{\bar{\alpha}z\}(1+|z|^2 - 2Re\{\bar{\alpha}z\})\right].
    	\end{split}
  	\end{equation*}
  \end{center}
Hence the proof.
\end{proof}

\begin{proposition}\label{conj}
The Berezin range of $C_{\phi_\alpha}$ on $A^2(\mathbb{D})$ is closed under complex conjugation and therefore symmetric about the real axis. 
\end{proposition}
\begin{proof}
	Put $z=re^{i\theta}$ and $\alpha = \rho e^{i\psi}$. We claim that $\widetilde{C_{\phi}}(re^{i\theta})=\overline{\widetilde{C_{\phi}}(re^{i(2\psi-\theta)})}$. This is the case if and only if
$$\left[\frac{1-r^2}{1-re^{-i\theta}\phi_\alpha(re^{i\theta})}\right]^2 = \overline{\left[\frac{1-r^2}{1-re^{-i(2\psi-\theta)}\phi_\alpha(re^{i(2\psi-\theta)})}\right]^2}.$$
Since both the terms a positive, taking the square root, we get
 $$\frac{1-r^2}{1-re^{-i\theta}\phi_\alpha(re^{i\theta})} = \overline{\frac{1-r^2}{1-re^{-i(2\psi-\theta)}\phi_\alpha(re^{i(2\psi-\theta)})}},$$
 if and only if $re^{-i\theta}\phi_\alpha(re^{i\theta}) = re^{i(2\psi-\theta)} \overline{\phi_\alpha(re^{i(2\psi-\theta)})}$ or, equivalently, if and only if $\phi_\alpha(re^{i\theta}) = e^{i(2\psi)} \overline{\phi_\alpha(re^{i(2\psi-\theta)})}$ So let us compute:
	\begin{center}
		\begin{equation*}
			\begin{split}
				e^{i(2\psi)}\overline{\phi_\alpha(re^{i(2\psi-\theta)})} &= e^{i(2\psi)}\frac{re^{i(\theta-2\psi)}-\rho e^{-i\psi}}{1-\rho e^{i\psi}re^{i(\theta-2\psi)}}\\
				&=\frac{re^{i\theta}-\rho e^{i\psi}}{1-\rho e^{-i\psi}re^{i\theta}}\\
				&=\phi_\alpha(re^{i\theta}).
			\end{split}
		\end{equation*}
	\end{center}
\end{proof}

\begin{corollary}
	If the Berezin range of $C_{\phi_\alpha}$ on $A^2(\mathbb{D})$ is convex, then $Re\{\widetilde{C}_{\phi_\alpha}(z)\} \in \text{Ber}(C_\phi)$ for each $z\in \mathbb{D}$.
\end{corollary}

\begin{proof}
	Suppose $\text{Ber}(C_\phi)$ is convex. Then from Proposition \ref{conj} $\text{Ber}(C_\phi)$ is closed under complex conjugation. Therefore we have
	\begin{center}
		$\frac{1}{2}\widetilde{C}_{\phi_\alpha}(z) + \frac{1}{2}\overline{\widetilde{C}_{\phi_\alpha}(z)} = Re\{\widetilde{C}_{\phi_\alpha}(z)\} \in \text{Ber}(C_{\phi_\alpha})$.
	\end{center}
\end{proof}

\begin{theorem}
The Berezin range of $C_{\phi_\alpha}$ on $A^2(\mathbb{D})$ is convex if and only if $\alpha=0.$
\end{theorem}
\begin{proof}
Suppose that $\alpha=0$. Then we have $\phi_\alpha(z) = z$ and 
$$\widetilde{C}_{\phi_\alpha}(z) = \frac{(1 - |z|^2)^2}{(1 - \overline{z}z)^2} = 1.$$
So $\text{Ber}(C_{\phi_\alpha})= \{1\}$, which is convex.
Conversely, assume that $\text{Ber}(C_{\phi_\alpha})$ is convex. From Proposition \ref{conj} we have $Re\{\widetilde{C}_{\phi_\alpha}(z)\} \in \text{Ber}(C_{\phi_\alpha})$. Therefore, for each $z\in \mathbb{D}$, we can find $w\in \mathbb{D}$ such that
$$\widetilde{C}_{\phi_\alpha}(w) = Re\{\widetilde{C}_{\phi_\alpha}(z)\}.$$
From this, we get

\begin{center}
\begin{equation*}
\begin{split}
Im\{\widetilde{C}_{\phi_\alpha}\}(w) &= 2k_{\alpha,w}^2 \left[(1-|w|^2)(1-Re\{\bar{\alpha}w\})+2(Im\{\bar{\alpha}w\})^2\right]\\
          &\qquad\times \left[Im\{\bar{\alpha}w\}(1+|w|^2 - 2Re\{\bar{\alpha}w\})\right]\\
          &=0
\end{split}
\end{equation*}
\end{center} 
where $k_{\alpha,w}$ is defined as in Lemma \ref{real}. Since $k_{\alpha,w}^2, (1+|w|^2 - 2Re\{\bar{\alpha}w\}), (1-|w|^2)$ and $\left[(1-Re\{\bar{\alpha}w\})+2(Im\{\bar{\alpha}w\})^2\right]$ are greater than zero for any $\alpha,w \in \mathbb{D}$, we have $Im\{\widetilde{C}_{\phi_\alpha}\}(w) = 0$ if and only if $Im\{\bar{\alpha}w\} = 0.$ This says that $\alpha$ and $w$ lie on a line passing through the origin. So we put $w=r\alpha$ for some $r \in (-1/|\alpha|,1/|\alpha|).$ Now we have
\begin{center}
		\begin{equation*}
			\begin{split}
				\widetilde{C}_{\phi_\alpha}(w) &= Re\{\widetilde{C}_{\phi_\alpha}(r\alpha)\}\\
				&=k_{\alpha,r\alpha}^2\left[(1-|r\alpha|^2)(1-Re\{\bar{\alpha}r\alpha\})+2(Im\{\bar{\alpha}r\alpha\})^2\right]^2\\
				&\quad - k_{\alpha,r\alpha}^2\left[Im\{\bar{\alpha}r\alpha\}(1+|r\alpha|^2 - 2Re\{\bar{\alpha}r\alpha\})\right]^2\\
				&=\left[\frac{(1-|r\alpha|^2)}{(1-|r\alpha|^2 + 2iIm\{\alpha\bar{r\alpha}\})}\right]^2\left[(1-|r\alpha|^2)(1-Re\{\bar{\alpha}r\alpha\})+2(Im\{\bar{\alpha}r\alpha\})^2\right]^2\\ 
				&\quad - \left[\frac{(1-|r\alpha|^2)}{(1-|r\alpha|^2 + 2iIm\{\alpha\bar{r\alpha}\})}\right]^2\left[Im\{\bar{\alpha}r\alpha\}(1+|r\alpha|^2 - 2Re\{\bar{\alpha}r\alpha\})\right]^2\\
				&=\frac{1}{(1-|r\alpha|^2)^2}\left[(1-|r\alpha|^2)(1-r|\alpha|^2)^2\right]^2\\
				&= (1-r|\alpha|^2)^4.
			\end{split}
		\end{equation*}
	\end{center}
Therefore, we get, $\{\widetilde{C}_{\phi_\alpha}(r\alpha) : r\in(-1/|\alpha|,1/|\alpha|)\} = \left((1-|\alpha|)^4,(1+|\alpha|)^4\right).$
Now put $z=\rho e^{i\theta}$, we can show that
\[
    \lim_{\rho\rightarrow 1^-} \widetilde{C}_{\phi_\alpha}(\rho e^{i\theta})= \begin{cases}
  0  &  \text{if}~~ \alpha\neq 0, \\
  1 & \text{if}~~ \alpha=0.
\end{cases}
  \]
 From this we can see that when $\alpha \neq 0$, given $\epsilon$ with $0<\epsilon< (1 - |\alpha|)^4$, there exist a point $z$ such that $|Re\{\widetilde{C}_{\phi_\alpha}(z)\}|<\epsilon$. But if $\widetilde{C}_{\phi_\alpha}(w) = Re\{\widetilde{C}_{\phi_\alpha}(z)\}$, we get a contradiction since $\widetilde{C}_{\phi_\alpha}(w)\in \left((1-|\alpha|)^4,(1+|\alpha|)^4\right).$ Thus, for $\alpha\neq 0$, $\text{Ber}(C_{\phi_\alpha})$ cannot be convex.
\end{proof}

\section{The Berezin set mapping theorem for some self-adjoint operators}

In this section, we use superquadratic functions to prove the Berezin set
mapping theorem for some selfadjoint operators on the reproducing kernel
Hilbert space.

Let $J\subseteq\mathbb{R=(}-\infty,+\infty\mathbb{)}$ be an interval. Recall
that a function $f:J\rightarrow\mathbb{R}$ is called convex if
\[
f(tx+(1-t)y)\leq tf(x)+(1-t)f(y)
\]
for all points $x,y\in J$ and all $t\in\left[  0,1\right]  .$ If $-f$ is
convex then we say that $f$ is concave. Moreover, if both convex and concave,
then $f$ is said to be affine.

\begin{definition}
\label{D1}\textbf{(}\cite{1}\textbf{) }A function $f:[0,\infty)\rightarrow
\mathbb{R}$ is superquadratic provided that for all $x\geq0$ there exists a
constant $C_{x}\in\mathbb{R}$ such that%
\begin{equation}
f(y)\geq f(x)+C_{x}(y-x)+f(\left\vert y-x\right\vert ) \label{1}%
\end{equation}
for all $y\geq0.$
\end{definition}

As observed in \cite{1}, if $f(x)$ is superquadratic and $a,b\geq0,$ then
$f(x)-(ax+b)$ is also superquadratic. We say that $f$ is subquadratic if $-f$
is superquadratic. Thus, for a superquadratic function we require that $f$ lie
above its tangent line plus a translation of $f$ itself.

Also remark that at first glance, condition (\ref{1}) appears to be stronger
than convexity but if $f$ takes negative values then it may be considerably
weaker. To emphasize just how poorly behaved superquadratic functions can be
we remark that any function $f$ satisfying $-2\leq f(x)\leq-1$ for all $x$ is
superquadratic. Just take $C_{x}=0$ in (\ref{1}).

Non-negative superquadratic functions are much better behaved as we see next
(see \cite{1}).

\begin{lemma}
\label{L1}Let $f$ be a superquadratic function with $C_{x}$ as in Definition
\ref{D1}.

(i) Then $f(0)\leq0.$

(ii) If $f(0)=f^{\prime}(0)=0,$ then $C_{x}=f^{\prime}(x)$ whenever $f$ if
differentiable at $x>0.$

(iii) If $f\geq0,$ then $f$ is convex and $f(0)=f^{\prime}(0)=0.$
\end{lemma}

The next result (see \cite{1}) gives a sufficient condition when convexity
(concavity) implies super (sub) quadraticity.

\begin{lemma}
\label{L2}If $f^{\prime}$ is convex (concave) and $f(0)=f^{\prime}(0)=0,$ then
it is super (sub) quadratic. The converse is not true.
\end{lemma}

Remark that subquadraticity does always not imply concavity, i.e., there
exists a subquadratic function which is convex. For example, $f(x)=x^{p},$
$x\geq0$ and $1\leq p\leq2$ is subquadratic and convex.

In 1906, Jensen in \cite{9} proved his famous characterization of convex
functions. Namely, for a continuous functions $f$ defined on a real interval
$J,$ $f$ is convex if and only if%
\[
f\left(\frac{x+y}{2}\right)\leq\frac{f(x)+f(y)}{2}%
\]
for all $x,y\in J.$

In 1965, a parallel characterization of Jensen convexity was presented by
Popoviciu \cite{23}, where he proved his celebrated inequality (named now
Popoviciu inequality in the literature), as follows:

\begin{theorem}
\label{T1}Let $f:J\rightarrow\mathbb{R}$ be a continuous function. Then $f$ is
convex if and only if%
\begin{equation}
\frac{2}{3}\left[  f\left(\frac{x+z}{2}\right)+f\left(\frac{y+z}{2}\right)+f\left(\frac{x+y}{2}\right)\right]
\leq f\left(\frac{x+y+z}{3}\right)+\frac{f(x)+f(y)+f(z)}{3} \label{2}%
\end{equation}
for all $x,y,z\in J,$ and the equality occured by $f(x)=x,$ $x\in J.$
\end{theorem}

In fact, Popoviciu characterization of a convex function is sound and several
mathematicians greatly received his work since that time and much of them
considered his characterization as an alternative approach to describe convex
functions. For instance, Popoviciu inequality can be considered as an elegant
generalization of Hlawka's inequality using convexity as a simple tool of
geometry. For more fact and application of Popoviciu inequality, see Popoviciu
\cite{23}, Niculescu and Popoviciu \cite{21} and Beucze, Niculescu and
Popoviciu \cite{5}, and for other related results see Mitrinovic, Pecaric and
Fink \cite{18}, Grinberg \cite{10} and Alomari \cite{3}.

In this section, we focus two operator versions of Popoviciu's inequality for
positive selfadjoint operators in reproducing kernel Hilbert spaces under
positive linear maps for both super (sub) quadratic and convex functions, and
prove the Berezin set mapping theorem (see Corollary \ref{C2}).

The relationship between the Berezin set $\mathrm{Ber}(A)$ of operator $A$ and
its spectrum $\sigma(A)$ is studied in \cite{15} for some concrete operators.
Since the so-called spectral mapping theorem of the form $\varphi
(\sigma(A))=\sigma(\varphi(A))$ plays a central role in many problems and
applications of operator theory, the same results for the Berezin set of
operators apparently will be also interesting and important. Here we do
apparently the first attempt in this direction and prove such theorem, which
looks as%
\[
f(\mathrm{Ber}(\Phi(A)))=\mathrm{Ber}(\Phi(f(A))),
\]
where $A$ is a positive selfadjoint operator on $\mathcal{H}\left(
\Omega\right)  ,$ $f$ is a superquadratic function and $\Phi:\mathcal{B}%
\left(  \mathcal{H}\left(  \Omega\right)  \right)  \mathcal{\rightarrow
}\mathcal{B}\left(  \mathcal{K(}Q\mathcal{)}\right)  $ is a normalized
positive linear map; $\mathcal{B}\left(  \mathcal{H}\left(  \Omega\right)
\right)  $ denotes the Banach algebra of all bounded linear operators on
$\mathcal{H}\left(  \Omega\right)  .$

Let $A$ be a selfadjoint linear operator on a complex Hilbert space
$(H;\left\langle .,.\right\rangle ).$ Recall that the Gelfand map establishes
a $\ast$-isometrically isomorphism $\Psi$ between the set $C(\sigma(A))$ of
all continuous functions defined on the spectrum of $A$ and the $C^{\ast}%
$-algebra $\mathcal{A}\left(  A\right)  $ generated by $A$ and the identity
operator $I_{H}$ on $H$ as follows (see for instance \cite[p. 3]{11}) :

For any two functions $f,g\in C(\sigma(A))$ and any numbers $\alpha,\beta
\in\mathbb{C}$ we have

(i) $\Psi(\alpha f+\beta g)=\alpha\Psi(f)+\beta\Psi(g);$

(ii) $\Psi(fg)=\Psi(f)\Psi(g)$ and $\Psi(\overline{f})=\Psi(f)^{\ast};$

(iii) $\left\Vert \Psi(f)\right\Vert =\left\Vert f\right\Vert :=\sup
_{t\in\sigma(A)}\left\vert f(t)\right\vert ;$

(iv) $\Psi(f_{0})=I_{H}$ and $\Psi(f_{1})=A,$ where $f_{0}(t)=1$ and
$f_{1}(t)=t,$ for $t\in\sigma(A).$

With this notation we define%
\[
f(A):=\Psi(f)\text{ for all }f\in C(\sigma(A))
\]
and we call it the continuous functional calculus for a selfadjoint operator
$A.$

If $A$ is a self-adjoint operator and $f$ is a real valued continuous function
on $\sigma(A),$ then $f(t)\geq0$ for any $t\in\sigma(A)$ implies that
$f(A)\geq0,$ i.e., $f(A)$ is a positive operator on $H.$ Moreover, if both $f$
and $g$ are real valued functions on $\sigma(A),$ then the following important
property holds:%
\begin{equation}
f(t)\geq g(t)\text{ for any }t\in\sigma(A)\text{ implies that }f(A)\geq g(A)
\label{3}%
\end{equation}
in the operator order of $\mathcal{B}\left(  H\right)  .$

 The linear map $\Phi:\mathcal{B}\left(  \mathcal{H}\right)
\mathcal{\rightarrow B}\left(  \mathcal{K}\right)  $ is positive if it
preserves the operator order, i.e., if $A\in\mathcal{B}^{+}(\mathcal{H})$ then
$\Phi(A)\in\mathcal{B}^{+}(\mathcal{K}).$ Obviously, a positive linear map
$\Phi$ preserves the order relation, namely $A\leq B\Rightarrow\Phi(A)\leq
\Phi(B)$ and preserves the adjoint operation $\Phi(A^{\ast})=\Phi(A)^{\ast}.$
Moreover, $\Phi$ is said to be normalized (unital) if it preserves the
identity operator, i.e., $\Phi(I_{\mathcal{H}})=I_{\mathcal{K}}.$

Our first result proves the operator version of the Popoviciu inequality for
superquadratic functions under positive linear maps. The proof uses the same
argument used in the proof of Theorem $2.2$ of the work \cite{3} (which we
omit); only for completeness we provide here some sketch of the proof. 

The Berezin transform of $T\in \mathcal{B}(\mathcal{H})$ is  defined as $\widetilde{T}(x) := \langle T\hat{k}_{x},\hat{k}_{x} \rangle_{\mathcal{H}}$. For notational convenience, we denote the Berezin transform $\widetilde{T}(x) = (T)^{\widetilde{•}}(x).$
\begin{theorem}
\label{T2}Let $\mathcal{H}=\mathcal{H}\left(  \Omega\right)  $and
$\mathcal{K}=\mathcal{K}\left(  Q\right)  $be two reproducing kernel Hilbert
spaces over the set $\Omega$ and $Q$ with the normalized reproducing kernels
$\widehat{k}_{\lambda,\mathcal{H}}$ and $\widehat{k}_{\mu,\mathcal{K}},$ and
let $A,B,C\in\mathcal{B}\left(  \mathcal{H}\right)  $ be three positive
selfadjoint operators, $\Phi:\mathcal{B}\left(  \mathcal{H}\right)
\mathcal{\rightarrow B}\left(  \mathcal{K}\right)  $ be a normalized positive
linear map. If $f:\left[  0,\infty\right)  \rightarrow\mathbb{R}$ is
continuous superquadratic, then
\begin{align*}
&  \left(  \Phi\left(  \frac{f(A)+f(B)+f(C)}{3}\right)  \right)
^{\widetilde{}}(\mu)+f\left(  \Phi\left(  \frac{A+B+C}{3}\right)
^{\widetilde{}}(\mu)\right)  \\
&  \geq\frac{2}{3}\left[  f\left(  \Phi\left(  \frac{A+B}{2}\right)
^{\widetilde{}}(\mu)\right)  +f\left(  \Phi\left(  \frac{B+C}{2}\right)
^{\widetilde{}}(\mu)\right)  +f\left(  \Phi\left(  \frac{A+C}{2}\right)
^{\widetilde{}}(\mu)\right)  \right]  \\
&  +\frac{1}{3}\left[  \left(  \Phi\left(  f\left\vert \left(  A-\Phi\left(
\frac{B+C}{2}\right)  ^{\widetilde{}}(\mu)I_{\mathcal{H}}\right)  \right\vert
\right)  \right)  ^{\widetilde{}}(\mu)\right.
\end{align*}%
\begin{align*}
&  +f\left(  \left\vert \left(  \Phi\left(  \frac{2A-B-C}{6}\right)  \right)
^{\widetilde{}}(\mu)\right\vert \right)  +\left(  \Phi\left(  f\left(
\left\vert C-\Phi\left(  \frac{A+B}{2}\right)  ^{\widetilde{}}(\mu
)I_{\mathcal{H}}\right\vert \right)  \right)  \right)  ^{\widetilde{}}(\mu)\\
&  +f\left(  \left\vert \left(  \Phi\left(  \frac{2C-A-B}{6}\right)  \right)
^{\widetilde{}}(\mu)\right\vert \right)  +\left(  \Phi\left(  f\left(
\left\vert B-\Phi\left(  \frac{A+C}{2}\right)  ^{\widetilde{}}(\mu
)I_{\mathcal{H}}\right\vert \right)  \right)  \right)  ^{\widetilde{}}(\mu)
\end{align*}%
\begin{equation}
\left.  +f\left(  \left\vert \left(  \Phi\left(  \frac{2B-A-C}{6}\right)
\right)  ^{\widetilde{}}(\mu)\right\vert \right)  \right]  \label{4}%
\end{equation}
for each $\mu\in Q.$
\end{theorem}

\begin{proof}
Since $f$ is superquadratic on $J,$ by utilizing the continuous functional
calculus for the operator $T\geq0,$ we have by property (\ref{3}) and
inequality (\ref{1}) that
\[
f(T)\geq f(x)I_{\mathcal{H}}+C_{x}(T-xI_{\mathcal{H}})+f(\left\vert
T-xI_{\mathcal{H}}\right\vert ),
\]
and since $\Phi$ is the normalized positive linear map, we get%
\[
\Phi(f(T))\geq f(x)I_{\mathcal{K}}+C_{x}\Phi\left(  T-xI_{\mathcal{H}}\right)
+\Phi\left(  f\left(  \left\vert T-xI_{\mathcal{H}}\right\vert \right)
\right)
\]
which implies that%
\begin{align}
\left\langle \Phi(f(T))\widehat{k}_{\mu,\mathcal{K}},\widehat{k}%
_{\mu,\mathcal{K}}\right\rangle  &  \geq f(x)\left\langle \widehat{k}%
_{\mu,\mathcal{K}},\widehat{k}_{\mu,\mathcal{K}}\right\rangle \nonumber\\
&  +C_{x}\left\langle \left[  \Phi\left(  T-xI_{\mathcal{H}}\right)  \right]
\widehat{k}_{\mu,\mathcal{K}},\widehat{k}_{\mu,\mathcal{K}}\right\rangle
\nonumber\\
&  +\left\langle \Phi\left(  f\left(  \left\vert T-xI_{\mathcal{H}}\right\vert
\right)  \right)  \widehat{k}_{\mu,\mathcal{K}},\widehat{k}_{\mu,\mathcal{K}%
}\right\rangle \label{5}%
\end{align}
for each $\mu\in Q.$

Let $A,B,C$ be three positive selfadjoint operators in $\mathcal{B}\left(
\mathcal{H}\right)  .$ Since $f$ is superquadratic, by applying (\ref{5}) for
the operator $A\geq0$ with $x_{1}=\Phi\left(  \frac{B+C}{2}\right)
^{\widetilde{}}(\mu),$ we get%
\begin{align}
\Phi\left(  f(A)\right)  ^{\widetilde{}}(\mu)  &  \geq f\left(  \Phi\left(
\frac{B+C}{2}\right)  ^{\widetilde{}}(\mu)\right)  +C_{x_{1}}\left(
\Phi\left(  \frac{2A-B-C}{2}\right)  \right)  ^{\widetilde{}}(\mu)\nonumber\\
&  +\left(  \Phi\left(  f\left(  \left\vert A-\Phi\left(  \frac{B+C}%
{2}\right)  ^{\widetilde{}}(\mu)\right\vert \right)  \right)  \right)
^{\widetilde{}}(\mu) \label{6}%
\end{align}
for each $\mu\in Q.$

Again applying (\ref{5}) for the operator $C\geq0$ with $x_{2}=\Phi\left(
\frac{A+B}{2}\right)  ^{\widetilde{}}(\mu),$ we have%
\begin{align}
\Phi\left(  f(C)\right)  ^{\widetilde{}}(\mu)  &  \geq f\left(  \Phi\left(
\frac{A+B}{2}\right)  ^{\widetilde{}}(\mu)\right)  +C_{x_{2}}\left(
\Phi\left(  \frac{2C-A-B}{2}\right)  \right)  ^{\widetilde{}}(\mu)\nonumber\\
&  +\left(  \Phi\left(  f\left(  \left\vert C-\Phi\left(  \frac{A+B}%
{2}\right)  ^{\widetilde{}}(\mu)I_{\mathcal{H}}\right\vert \right)  \right)
\right)  ^{\widetilde{}}(\mu) \label{7}%
\end{align}
for each $\mu\in Q.$ Also, for the operator $B\geq0$ with $x_{3}=\Phi\left(
\frac{A+C}{2}\right)  ^{\widetilde{}}(\mu),$ we get%
\begin{align}
\Phi\left(  f(B)\right)  ^{\widetilde{}}(\mu)  &  \geq f\left(  \Phi\left(
\frac{A+C}{2}\right)  ^{\widetilde{}}(\mu)\right)  +C_{x_{3}}\left(
\Phi\left(  \frac{2B-A-C}{2}\right)  \right)  ^{\widetilde{}}(\mu)\nonumber\\
&  +\left(  \Phi\left(  f\left(  \left\vert B-\Phi\left(  \frac{A+C}%
{2}\right)  ^{\widetilde{}}(\mu)I_{\mathcal{H}}\right\vert \right)  \right)
\right)  ^{\widetilde{}}(\mu) \label{8}%
\end{align}
for each $\mu\in Q.$

We have from (\ref{6}), (\ref{7}), (\ref{8}) that%
\begin{align*}
&  \left(  \Phi\left(  \frac{f(A)+f(B)+f(C)}{3}\right)  \right)
^{\widetilde{}}(\mu)\\
&  \geq\frac{1}{3}\left[  f\left(  \Phi\left(  \frac{A+B}{2}\right)
^{\widetilde{}}(\mu)\right)  +f\left(  \Phi\left(  \frac{B+C}{2}\right)
^{\widetilde{}}(\mu)\right)  +f\left(  \Phi\left(  \frac{A+C}{2}\right)
^{\widetilde{}}(\mu)\right)  \right]  \\
&  +\frac{1}{3}\left[  C_{x_{1}}\left(  \Phi\left(  \frac{2A-B-C}{2}\right)
\right)  ^{\widetilde{}}(\mu)+C_{x_{2}}\left(  \Phi\left(  \frac{2C-A-B}%
{2}\right)  \right)  ^{\widetilde{}}(\mu)\right.  \\
&  \left.  +C_{x_{3}}\left(  \Phi\left(  \frac{2B-A-C}{2}\right)  \right)
^{\widetilde{}}(\mu)\right]  \\
&  +\left(  \Phi\left(  f\left(  \left\vert C-\Phi\left(  \frac{A+B}%
{2}\right)  ^{\widetilde{}}(\mu)I_{\mathcal{H}}\right\vert \right)  \right)
\right)  ^{\widetilde{}}(\mu)\\
&  +\left(  \Phi\left(  f\left(  \left\vert A-\Phi\left(  \frac{B+C}%
{2}\right)  ^{\widetilde{}}(\mu)I_{\mathcal{H}}\right\vert \right)  \right)
\right)  ^{\widetilde{}}(\mu)\\
&  +\left(  \Phi\left(  f\left(  \left\vert B-\Phi\left(  \frac{A+C}%
{2}\right)  ^{\widetilde{}}(\mu)I_{\mathcal{H}}\right\vert \right)  \right)
\right)  ^{\widetilde{}}(\mu),
\end{align*}
which finally implies the required inequality (\ref{4}) (see the proof of
Theorem 2.2. in \cite{3}).
\end{proof}

The following result (see Corollary \ref{C1}, (ii) below) gives, in
particular, a refinement of the main property of superquadratic functions that
$-f(0)\geq0$ (see Lemma \ref{L1}, (i)).

\begin{corollary}
\label{C1}Let $A\in\mathcal{B(H}\left(  \Omega\right)  \mathcal{)}$ be a
positive self-adjoint operator and $\Phi:\mathcal{B}\left(  \mathcal{H}\left(
\Omega\right)  \right)  \mathcal{\rightarrow}\mathcal{B(K(}Q\mathcal{))}$ be a
normalized positive linear map. Let $f:\left[  0,\infty\right)  \rightarrow
\mathbb{R}$ be a continuous superquadratic function. Then:

(i)
\begin{equation}
-f(0)\geq f\left(  \left(  \Phi\left(  A\right)  \right)  ^{\widetilde{}%
}\left(  \mu\right)  \right)  -\Phi\left(  f(A)\right)  ^{\widetilde{}}\left(
\mu\right)  +\left(  \Phi\left(  f\left(  \left\vert A-\left(  \Phi(A)\right)
^{\widetilde{}}(\mu)I_{\mathcal{H}}\right\vert \right)  \right)  \right)
^{\widetilde{}}(\mu) \label{16}%
\end{equation}
for all $\mu\in Q;$

(ii)
\begin{align}
-f(0) &  \geq\sup_{\mu\in Q}\left[  f\left(  \left(  \Phi\left(  A\right)
\right)  ^{\widetilde{}}\left(  \mu\right)  \right)  -\left(  \Phi\left(
f(A)\right)  \right)  ^{\widetilde{}}\left(  \mu\right)  +\right.  \nonumber\\
&  +\left.  \left(  \Phi\left(  f\left(  \left\vert A-\widetilde{\Phi(A)}%
(\mu)I_{\mathcal{H}}\right\vert \right)  \right)  \right)  ^{\widetilde{}}%
(\mu)\right]  .\label{17}%
\end{align}

\end{corollary}

\begin{proof}
Setting $C=B=A$ in (\ref{4}), we get (\ref{16}). Inequality (\ref{17}) follows
from (\ref{16}).
\end{proof}

The following corollary means the Berezin set mapping theorem.

\begin{corollary}
\label{C2}Let $A,\Phi$ and $f$ be the same as in Corollary \ref{C1} such that
$f$ is nonnegative and%
\begin{equation}
f\left(  \left(  \Phi(A)\right)  ^{\widetilde{}}\left(  \mu\right)  \right)
\geq\Phi\left(  f(A)\right)  ^{\widetilde{}}\left(  \mu\right)  ,\text{
}\forall\mu\in Q.\label{18}%
\end{equation}
Then $f$ is convex and%
\begin{equation}
f(\mathrm{Ber}(\Phi(A)))=\mathrm{Ber}(\Phi(f(A))).\label{19}%
\end{equation}

\end{corollary}

\begin{proof}
Since $A\geq0,$ $A^{\ast}=A,\Phi$ is a positive linear map, $f\in
C(\sigma(A))$ is nonnegative and $\left\vert A-\left(  \Phi(A)\right)
^{\widetilde{}}\left(  \mu\right)  I_{\mathcal{H}}\right\vert \geq0,$ by Lemma
\ref{L1} we deduce that $f$ is convex and $f\left(  \left\vert A-\left(
\Phi(A)\right)  ^{\widetilde{}}\left(  \mu\right)  I_{\mathcal{H}}\right\vert
\right)  \geq0$ and $f(0)=0.$ Then we have from (\ref{16}) that%
\[
0\geq f\left(  \left(  \Phi(A)\right)  ^{\widetilde{}}\left(  \mu\right)
\right)  -\left(  \Phi\left(  f(A)\right)  \right)  ^{\widetilde{}}\left(
\mu\right)  ,
\]
or equivalently%
\begin{equation}
\Phi\left(  f(A)\right)  ^{\widetilde{}}\left(  \mu\right)  \geq f\left(
\Phi(A)^{\widetilde{}}\left(  \mu\right)  \right)  ,\text{ }\forall\mu\in
Q.\label{20}%
\end{equation}
Now (\ref{18}) and (\ref{20}) imply that $f\left(  \Phi(A)^{\widetilde{}%
}\left(  \mu\right)  \right)  =\Phi\left(  f(A)\right)  ^{\widetilde{}}\left(
\mu\right)  $ for all $\mu\in Q,$ which gives the desired formula (\ref{19}).
The proof is completed.
\end{proof}

Our next results give new estimates for the Berezin number of some operators.

\begin{proposition}
\label{P1}Let $A:\mathcal{H}\left(  \Omega\right)  \mathcal{\rightarrow
H}\left(  \Omega\right)  $ be a positive self-adjoint operator, $\Phi
:\mathcal{B}\left(  \mathcal{H}\left(  \Omega\right)  \right)
\mathcal{\rightarrow B(K(}Q\mathcal{))}$ be a normalized positive linear map,
and $f:\left[  0,\infty\right)  \rightarrow\mathbb{R}$ be a non negative and
superquadratic function. Then $f$ is convex and%
\[
\mathrm{ber}(\Phi(f(A)))\geq\sup_{\mu\in Q}f\left(  \Phi(A)^{\widetilde{}%
}\left(  \mu\right)  \right)  .
\]

\end{proposition}

\begin{proof}
By considering that $f$ is non-negative superquadratic, we have by Lemma
\ref{L1} that $f$ is convex and%
\[
\Phi\left(  f(A)\right)  ^{\widetilde{}}\left(  \mu\right)  \geq f\left(
\Phi\left(  A\right)  ^{\widetilde{}}\left(  \mu\right)  \right)
\]
for all $\mu\in Q.$ This shows that%
\[
\mathrm{ber}(\Phi(f(A)))\geq\sup_{\mu\in Q}\left[  f\left(  \Phi
(A)^{\widetilde{}}\left(  \mu\right)  \right)  \right]  ,
\]
as desired.
\end{proof}

\begin{proposition}
\label{P2}Let $A\in\mathcal{B}\left(  \mathcal{H}\left(  \Omega\right)
\right)  $ be the self-adjoint operator, $\Phi:\mathcal{B}\left(
\mathcal{H}\left(  \Omega\right)  \right)  \mathcal{\rightarrow B(K(}%
Q\mathcal{))}$ be a normalized positive linear map and $f:\left[
0,\infty\right)  \rightarrow\mathbb{R}$ be the differentiable function with
$f(0)=f^{\prime}(0)=0.$ If $f^{\prime}$ is convex (concave), then $f$ is super
(sub)quadratic and%
\begin{equation}
\mathrm{ber}(\Phi(f^{^{\prime}}(A)))\geq(\leq)\sup_{\mu\in Q}\left[
f^{\prime}\left(  \Phi(A)^{\widetilde{}}\left(  \mu\right)  \right)  \right]
.\label{21}%
\end{equation}

\end{proposition}

\begin{proof}
The assertion that $f$ is superquadratic, follows from Lemma \ref{L2}. To
obtain the inequality (\ref{21}) we apply the same method used in the proof of
Theorem \ref{T2} by applying the inequality (see \cite[formula $(1.2)$]{3})%
\[
f(y)\geq f(x)+\alpha(y-x),\text{ }\forall y\in J,
\]
for $f^{^{\prime}}$ instead of $(1)$ for $f,$ so that we get the desired result.
\end{proof}

\begin{proposition}
\label{P3}Let $A,\Phi$ be the same as in Proposition \ref{P2} and $f:\left[
0,\infty\right)  \rightarrow\mathbb{R}$ be the continuous function. If $f$ is
convex (concave) and $f(0)=0,$ then%
\[
\mathrm{ber}(\Phi(f(A)))\geq(\leq)\sup_{\mu\in Q}\left[  f\left(
\Phi(A)^{\widetilde{}}\left(  \mu\right)  \right)  \right]  .
\]

\end{proposition}

\begin{proof}
By using Corollary $2.7$ of the paper \cite{3} for $F(x)=\int_{0}%
^{x}f(t)dt,x\in\left[  0,\infty\right)  ,$ it can be easily observed that
$F(0)=F^{\prime}(0)=0$ and $F^{\prime}(x)=f(x)$ is convex (concave) for all
$x\in\left[  0,\infty\right)  .$ So, Proposition \ref{P2} works. This
completes the proof.
\end{proof}

\textbf{Declaration of competing interest}

There is no competing interest.\\

\textbf{Data availability}

No data was used for the research described in the article.\\

{\bf Acknowledgments.} The first author is supported by the Junior Research Fellowship (09/0239(13298)/2022-EM) of CSIR (Council of Scientific and Industrial Research, India). The second author was supported by the Researchers Supporting Project number(RSPD2023R1056), King Saud University, Riyadh, Saudi Arabia. The third author is supported by the Teachers Association for Research Excellence (TAR/2022/000063) of SERB (Science and Engineering Research Board, India). First author and third author would like to thank Jaydeb Sarkar  for the valuable discussions.

\nocite{*}
\bibliographystyle{amsplain}
\bibliography{database}  
\end{document}